\documentclass[12pt,reqno]{amsart}
\usepackage{amsmath,amssymb,amsthm,bbm,fullpage,cite,verbatim}
\usepackage{amssymb}
\usepackage[utf8]{inputenc}
\usepackage{bbm}

\DeclareMathOperator*{\E}{\mathbb{E}}
\DeclareMathOperator*{\Prob}{\mathbb{P}}

\newtheorem{theorem}{Theorem}
\newtheorem{lemma}[theorem]{Lemma}

\newtheorem{conjecture}[theorem]{Conjecture}

\theoremstyle{definition}

\renewcommand\leq{\leqslant}
\renewcommand\geq{\geqslant}
\renewcommand\le{\leqslant}
\renewcommand\ge{\geqslant}

\begin{document}

\title{A Strengthening of Freiman's $3k-4$ Theorem}
\author{B\'ela Bollob\'as \and Imre Leader \and Marius Tiba}

\address{Department of Pure Mathematics and Mathematical Statistics,
Wilberforce Road,
Cambridge, CB3 0WA, UK, and Department of Mathematical Sciences,
University of Memphis, Memphis, TN 38152, USA}\email{b.bollobas@dpmms.cam.ac.uk}

\address{Department of Pure Mathematics and Mathematical Statistics,
Wilberforce Road, Cambridge, CB3 0WA, UK}\email{i.leader@dpmms.cam.ac.uk}

\address{Department of Pure Mathematics and Mathematical Statistics,
Wilberforce Road, Cambridge, CB3 0WA, UK}\email{mt576@cam.ac.uk}

\thanks{The first author was partially supported by NSF grant DMS-1855745}

\begin{abstract}

In its usual form, Freiman's $3k-4$ theorem states that if $A$ and $B$ are subsets of ${\mathbb Z}$ of
size $k$ with small sumset (of size close to $2k$) then
they are very close to arithmetic progressions. Our aim in this paper
is to strengthen this by allowing only a bounded number of possible summands
from one of the sets. We show that if $A$ and $B$ are subsets of ${\mathbb Z}$ of size $k$ such that
for any four-element subset $X$ of $B$ the sumset $A+X$ has size not much more than
$2k$ then already this implies that
$A$ and $B$ are very close to arithmetic progressions.

\end{abstract}

\maketitle

\section{Introduction}\label{intro}

Starting with the classical Cauchy-Davenport theorem \cite{Cauchy, Dav1, Dav2}, the sizes of sumsets in
Abelian groups have been studied in a host of papers. A new direction in this study was introduced in \cite{BLT}:
given sets of integers $A$ and $B$, say, does $B$ have a small subset $B'$ such that $|A+B'|$ is large, perhaps
even comparable to $|A+B|$? In particular, in \cite{BLT} the following result was proved.

\begin{theorem}\label{intro_three_translates_Z}
Let $A$ and $B$ be finite non-empty subsets of $\mathbb{Z}$ with $|A| \geq |B|$. Then there exist elements
$b_1,b_2,b_3 \in B$ such that
$$|A+\{b_1,b_2,b_3\}| \ge |A|+|B|-1.$$
\end{theorem}

Freiman's
$3k-4$ theorem (see
Freiman \cite{Frei-59,Frei-62} and also 
Lev and
Smeliansky~\cite{LevSme} and Stanchescu~\cite{Stan}), states the following, in the case where the
sets have the same size.

{\em Let $A$ and $B$ be finite subsets of ${\mathbb Z}$ with $|A|=|B|=k$ such that $|A+B|=2k-1+r$, where
$r \leq k-3$. Then there are arithmetic progressions $P$ and $Q$ of the same common difference,
containing $A$ and $B$ respectively, such that each has size $k+r$.}

A result of this kind is often called an `inverse' theorem, as it describes what happens when an inequality is close to being tight. (These are also often known as `stability'
results.)

What about a `bounded sumset' version of Freiman's $3k-4$ theorem?
Our aim in this paper is to show that indeed we can weaken the condition that
$A+B$ is small to just a condition that all sumsets of $A$ with a bounded number of terms of $B$ are 
small. This is really rather surprising. One could view this
as a kind of inverse result to Theorem~\ref{intro_three_translates_Z}. 

\begin{theorem}\label{linear_stability_four_translates}
There are constants  $c, \varepsilon_0>0$ such that for all $\varepsilon <  \varepsilon_0$ the
following holds. Let $A$ and $B$ be finite non-empty
subsets of $\mathbb{Z}$ with $n=|A| = |B|$. Suppose that for any four elements
$b_1,b_2,b_3,b_4 \in B$
we have $$|A+\{b_1, b_2,b_3,b_4\}| \leq (2 + \varepsilon)n-1.$$ Then there are arithmetic progressions
$P,Q$ in $\mathbb{Z}$ of size $(1+\varepsilon+c \varepsilon^2)n$ with the same common difference
such that $ B \subset Q$ and $$ |A \Delta P| \leq c \varepsilon n.$$
\end{theorem}

\vspace{5pt} The key point in this result is that the sizes of the progressions $P$ and $Q$ are `about'
$(1 + \varepsilon)n$: the precise form of the error term in this, as quadratic in $\varepsilon$, is not so important. 
Note that the dependence of the error terms (the sizes of $|A \Delta P|$ and $|B \Delta Q|$),
as linear functions of $\varepsilon$, is best possible. This may be seen by taking $B$ to be
an interval of length $n$ and $A$ to consist of of an interval of length $(1-2\varepsilon)n$ together with, on either
side of it, random half-sized subsets of the adjacent intervals of length $2 \varepsilon n$.

\vspace{5pt} We remark that the condition that $A$ has small symmetric difference with $P$ cannot 
be strengthened to insist that $A$ lies inside $P$. This is because one may always have a few
`rogue' points in $A$, far from the rest of $A$. The same would not apply to $B$, because adding a
few faraway points to $B$ is a great help in selecting the four points.

\vspace{5pt}
It would be extremely interesting to decide whether this result remains valid if we consider three
elements instead of four. To elaborate on this, the way that Theorem 1 is proved is by taking the
three points of $B$ as follows: we first take the smallest and
largest points of $B$ (the reader will see that we
{\it have} to take those points, as may be seen from the case when the two sets are intervals), and 
then choosing the third point at random from the remaining points of $B$. It is very surprising that
these two different ingredients mesh together so well, to give exactly the lower bound required.
However, and this is the key point, this approach {\it cannot} be   
used to prove a three-element version of Theorem~\ref{linear_stability_four_translates}. Indeed, this
may be seen by taking $B=[0,n]$ and $A=X \cup (X+n)$ where $X$ is a random set of density $1/2$ in an
interval of length $n$.     

\vspace{10pt}
The plan of the paper is as follows. In Section~2 we mention some background results from \cite{BLT} that we will need. We do give a brief summary of the proof of
Theorem~\ref{intro_three_translates_Z}, to help make the paper self-contained and also because the methods
there form a good `toy case' of the arguments in our main result. Indeed, those background results
from \cite{BLT} that we do not prove here may be proved by methods very similar to those in the proof
of Theorem~\ref{intro_three_translates_Z}.
Then in Section~3 we present our main result. 

Our notation is standard. To make our paper more readable, we often omit integer-part signs when these do not affect the
argument.

\vspace{5pt}
Sometimes we write `$x \mod d$' as shorthand for the infinite arithmetic progression
$\{ y \in \mathbb{Z}: y \equiv x \mod d \}$, and refer to it as a {\em fibre} mod $d$.
When $S$ is a subset of $\mathbb{Z}$ we often write $S^x$ for the intersection of this fibre with
$S$ -- when the value of $d$ is clear. (We sometimes write $S^x$ as $S^x_d$ when we want to stress the value of $d$.)
Thus $S^x=S\cap \pi^{-1}(x)$, where $\pi = \pi_d$ denotes
the natural projection from $\mathbb{Z}$ to $\mathbb{Z}_d$. We also write $\widetilde{S}$
for $\pi_d(S)$.

\vspace{5pt}
When we write a probability or an expectation over a finite set, we always assume that the elements
of the set are being sampled uniformly. Thus, for example,
for a finite set $X\subset \mathbb{Z}$ we denote the expectation and probability when we sample uniformly over
all $x \in X$ by respectively $\E_{x\in X} \text{ and } \Prob_{x\in X}.$
%
%

\vspace{5pt}
For more general background on sumsets, see the survey of Breuillard, Green and Tao \cite{BGT-doubling} or the
books of Nathanson \cite{Nathbook} or Tao and Vu \cite{taobook}. For extensions of Freiman's $3k-4$ theorem
to other settings, such as $\mathbb{Z}_p$ in particular, see Grynkiewicz~\cite{grynkiewicz-2}.

\vspace{5pt}
To end the Introduction, we give a brief overview of the proof of Theorem~\ref{linear_stability_four_translates}.
With $B$ having first element $0$ and last element $m$, we start by showing that (unless we are done) $A$ must have a
large projection
onto ${\mathbb Z}_m$. This is accomplished by a careful analysis of what happens when we take our four points
of $B$ to be $0$, $m$ and two random points: it is the interaction of the two random points that is critical
here. From this we find that $B$ is contained in a quite short arithmetic progression, and in turn this will imply
that $A$ is also relatively close to an arithmetic progression. However, these progressions are not small enough to
give Theorem~\ref{linear_stability_four_translates}, and so we need an argument that `boosts' this. This is a very
delicate analysis of, again, how the two random translates interact with each other and with the two fixed
translates. In fact, it is rather surprising that this boosting argument actually works: in length it is
most of the proof. (It is Theorem 9 below).

\section{Prerequisites}

\vspace{7pt}
\noindent
We start by giving an indication of the proof of Theorem \ref{intro_three_translates_Z}.
\vspace{5pt}
As explained above, the main idea is to fix two elements as the first and last elements of $B$ and then to select
the remaining element at random.

\vspace{5pt}
\noindent
Let $B$ have first element $0$ and last element $m$. Then $A+\{0, m\}=A\cup (A+m)$ satisfies
$\pi_m(A)=\pi_m\big(A\cup (A+m)\big)= \widetilde{A}$, and
\begin{equation}\label{new-sum-eq}
\bigg|A\cup (A+m) \bigg| \geq |A|+|\widetilde{A}|.
\end{equation}
To see this, note that $(A+m)^x=A^x+m$ for every $x\in {\mathbb Z}$. Hence, if $x\in \widetilde{A}=\pi_m(A)$ then
$(A+m)^x=A^x+m\ne A^x$,
so
$
\bigg| \big[A\cup (A+m)\big]^x\bigg|=\bigg|A^x\cup (A+m)^x \bigg| \geq |A^x|+1.
$
Consequently,
\[
\bigg|A\cup (A+m)\bigg| \geq \sum_{x\in \widetilde{A}} \bigg|[A\cup (A+m)]^x\bigg| \geq \sum_{x\in \widetilde{A}} |A^x|+1 = |A|+|\widetilde{A}|,
\]
completing the proof of \eqref{new-sum-eq}.

We mention that the way that the proof of Theorem~\ref{intro_three_translates_Z} in \cite{BLT} now proceeds is to
consider the excess contribution to $|A\cup (A+m)|$ that comes when we form the union with $A+b$, where $b$ is chosen
uniformly at random from the remaining points of $B$. One finds that this has expectation at least $|A|$ times
$1 - |\widetilde{A}| / |\widetilde{B}|$. If this is negative then the bound \eqref{new-sum-eq} is already enough to
finish the proof, while if it is positive then it is easy to check that together with \eqref{new-sum-eq} it gives
the required bound.

This actually gives the following
stronger version of Theorem~\ref{intro_three_translates_Z}. Let $A$ and $B$ be finite non-empty subsets of integers with $|A|\geq |B|$ and $B$ having smallest element $0$ and greatest element $m$.  Then we have
\begin{equation}\label{thm1.0.0}
     \max_{b_1,b_2,b_3 \in B}\bigg|A+\{b_1,b_2,b_3\}\bigg| \geq \E_{b\in B\setminus \{m\}} \bigg| A +\{0,b,m\} \bigg| \geq |A|+|B|-1.
\end{equation}

\vspace{5pt}
We now pass to some related results from \cite{BLT}. These are proved along broadly
similar lines to
the above results. The first one is about the case of equality.

\begin{theorem}\label{equality_four_translates}
  Let $A$ and $B$ be finite non-empty subsets of $\mathbb{Z}$ with $|A| = |B|$, with $\min(B) = 0$ and $\max(B)=m$.
  Suppose that when we choose an element $b$ of $B \setminus \{0,m\}$ uniformly at random we have
\begin{equation}\label{thm23.0.0}
    \E_{b \in B \setminus \{0,m\}} \bigg|A+\{0,b,m\}\bigg| \leq |A|+|B|-1.
\end{equation}
Then $A$ and $B$ are arithmetic progressions with the same common difference.

\vspace{7pt}
In particular, suppose that when we choose any three elements $b_1, b_2, b_3$ of $B$ we have
\[
\bigg|A+\{b_1,b_2,b_3\}\bigg| \le |A|+|B|-1.
\]
Then $A$ and $B$ are arithmetic progressions with the same common difference.
\end{theorem}

Then we need another version of Theorem \ref{intro_three_translates_Z}.

\begin{theorem}\label{technical_three_translates_Z}
Let $A$ and $B$ be finite non-empty subsets of $\mathbb{Z}$, with $\min(B) = 0$ and $\max(B)=m$. Then
$$\E_{b\in B \setminus \{m\}}\bigg|A+\{0,b,m\} \bigg| \ge |A|+|\pi_m(A)|+|A|\max\bigg(0, \frac{|B|-1-|\pi_m(A)|}{|B|-1} \bigg).$$
\end{theorem}

We note also a simple variant of the ideas above.

\begin{lemma}\label{lem8.2} Let $A$, $B$ and $A_1$ be finite non-empty subsets of $\mathbb{Z}$, with $\min(B)=0$ and $\max(B)=m$. Then with
  $\widetilde{A}=\pi_m(A)$ and $\widetilde{B}=\pi_m(B)$ we have 
$$\E_{b\in B \setminus \{m\}} \bigg|(A_1+b) \setminus \pi_m^{-1}(\widetilde{A})\bigg| \geq
|A_1| \max\bigg(0,\frac{|\widetilde{B}|-|\widetilde{A}|}{|\widetilde{B}|}\bigg).$$
\end{lemma}

\vspace{7pt}
The next result follows easily from Lemma~\ref{lem8.2}.

\begin{lemma}\label{residue.vs2}
Let $A$ and $B$ be finite non-empty subsets of $\mathbb{Z}$, with $\min(B)=0$ and $\max(B)=m$. Suppose that $|B|-1=|\widetilde{B}| \geq 8|\widetilde{A}|$, where $\widetilde{A}=\pi_m(A)$ and $\widetilde{B}=\pi_m(B)$. Then
\[
\E_{b_2,b_3 \in B\setminus\{m\}} \bigg|A+\{0,b_2,b_3\}  \bigg|\geq 2.5|A|.
\]
\end{lemma}

The last result from \cite{BLT} that we need is a strengthening of 
Theorem \ref{intro_three_translates_Z}.

\begin{theorem}\label{strengthening_of_three_translates}
Let $A$ and $B$ be finite non-empty subsets of $\mathbb{Z}$ with $|A| = |B|$, with $\min(B) = 0$ and $\max(B)=m$. Then
$$\E_{b\in B \setminus \{m\}}\bigg|A+\{0,b,m\}\bigg| \ge |A|+|B|-1+\max\bigg(0,\frac{(2|\pi_m(A)|-m)(
m-(|B|-1))-1}{|B|-1} \bigg).$$
\end{theorem}

\section{The Main Result}

In this section we prove Theorem~\ref{linear_stability_four_translates}, our stability version of Theorem~\ref{intro_three_translates_Z}. As we mentioned earlier,
the dependence of the error terms on $\varepsilon$, being linear, is of
best possible order. It will turn out that a key ingredient is actually the following weaker
version of Theorem \ref{linear_stability_four_translates}.

\vspace{5pt}
\begin{theorem}\label{weak_linear_stability_four_translates}
For any $\delta>0$ there exists $\varepsilon>0$ such that the following holds. Let $A$ and $B$ be finite non-empty subsets of $\mathbb{Z}$ with $n=|A| = |B|$. Suppose that
\begin{equation}\label{eq13.0}
    \max_{b_1,b_2,b_3,b_4 \in B}\bigg|A+\{b_1,b_2,b_3,b_4\}\bigg|\leq 2n-1 + \varepsilon n.
\end{equation}
Then there exist arithmetic progressions $P$ and $Q$ in $\mathbb{Z}$ with the same common difference and size $\lfloor(1+\delta)n\rfloor$ such that $|A \setminus P| \leq \delta n \text{ and } B \subset Q.$
\end{theorem}

\vspace{5pt}
We will also need the following technical result.
\begin{theorem}\label{stab_technical}
  For all sufficiently small $\alpha,\beta, \varepsilon>0$, taking
  $\nu=(\varepsilon+\varepsilon^2(1-\alpha)^{-1})(1-4\alpha-\beta)^{-1}$
  and $\mu=2^{15}(\varepsilon+\beta)$ the following holds. Let $A$ and $B$ be finite non-empty subsets of $\mathbb{Z}$ with $n=|A|=|B|$. Suppose that there exist integer arithmetic progressions $P$ and $Q$ with the same common difference $d$ and sizes $\lfloor (1+\alpha)n \rfloor$ and $\lfloor (1+\beta)n \rfloor$ such that
\begin{equation}\label{eq18.-2}
    |A\setminus P|\leq \alpha n \text{ and } B \subset Q
\end{equation}
and
\begin{equation}\label{eq18.-1}
    \max_{b_1,b_2,b_3,b_4 \in B}\bigg|A+\{b_1,b_2,b_3,b_4\}\bigg|\leq 2n-1 + \varepsilon n.
\end{equation}
Then there exist integer arithmetic progressions $P'$ and $Q'$ with common difference $d$ and sizes $\lfloor (1+\mu)n \rfloor$ and $\lfloor (1+\nu)n \rfloor$ such that $ |A\setminus P'|\leq \mu n \text{ and } B \subset Q'.$
\end{theorem}

\begin{proof}
  Fix $\alpha,\beta, \varepsilon>0$ sufficiently small. By Theorem~\ref{equality_four_translates}, we may assume
  that
\begin{equation}\label{eq18.0}
    \varepsilon n \geq 1.
\end{equation}
Set $R= d\mathbb{Z}$,
and assume without loss of generality that
$P, Q\subset R$.
Form the partition
$ A= A_1\sqcup A_2$
given by
$ A_1= A\cap R \text{ and } A_2=A\cap R^c$
and write
$n_1=|A_1| \text{ and } n_2=|A_2|.$
By \eqref{eq18.-2} we have
\begin{equation}\label{eq18.1}
    n\geq n_1\geq (1-\alpha)n \geq 9n_2.
\end{equation}
Assume that
$\min(B)=0 \text{ and }\max(B)=md.$
Take any subset
$ B_1\subset B$
of size
$ n_1=|B_1|$
such that
$\min(B_1)=0 \text{ and }\max(B_1)=md.$
On the one hand, by construction we have
$(A_1+B_1)\cap (A_2+B_1)=\emptyset.$
On the other hand, by \eqref{eq18.-1} we have
$$ {\mathbb E}_{b_2,b_3\in B_1\setminus \{md\}}\bigg|A+\{0,b_2,b_3,md\}\bigg|\leq 2n-1 + \varepsilon n.$$
The last two relations imply
\begin{equation*}
    {\mathbb E}_{b_2\in B_1\setminus \{md\}}\bigg|A_1+\{0,b_2,md\}\bigg| + {\mathbb E}_{b_2,b_3\in B_1\setminus \{md\}}\bigg|A_2+\{0,b_2,b_3\}\bigg|\leq 2n-1 + \varepsilon n.
\end{equation*}
By Theorem~\ref{technical_three_translates_Z}, we have
$$ {\mathbb E}_{b_2\in B_1\setminus \{md\}}\bigg|A_1+\{0,b_2,md\}\bigg| \geq 2n_1-1.$$
Lemma~\ref{residue.vs2} and \eqref{eq18.1} tell us that
\begin{equation*}
    {\mathbb E}_{b_2,b_3\in B_1\setminus \{md\}}\bigg|A_2+\{0,b_2,b_3\}\bigg|\geq 2.5n_2.
\end{equation*}
The last three inequalities imply
\begin{equation}\label{eq18.2}
    {\mathbb E}_{b_2\in B_1\setminus \{md\}}\bigg|A_1+\{0,b_2,md\}\bigg| \leq 2n_1-1 + \varepsilon n_1.
\end{equation}
and also
\begin{equation}\label{eq18.25}
    n_2\leq 2\varepsilon n \text{ i.e. } n_1\geq (1-2\varepsilon)n.
\end{equation}
By \eqref{eq18.-1} and \eqref{eq18.25} we have
\begin{equation}\label{eq18.27}
    \max_{b_2,b_3 \in B_1\setminus\{md\}}\bigg|A_1+\{0,b_2,b_3,md\}  \bigg| \leq 2n_1-1+8\varepsilon n_1.
\end{equation}
For the rest of the proof we focus on the sets $A_1$ and $B_1$ and thus we may assume without loss of
generality that $d=1.$
\\

\textbf{The construction of $Q'$.}
By Theorem~\ref{strengthening_of_three_translates} we have
\begin{equation}\label{eq18.3}
    {\mathbb E}_{b_2\in B_1\setminus \{md\}}\bigg|A_1+\{0,b_2,m\}\bigg|\geq 2n_1-1+\frac{(2|\pi_m(A_1)|-m)(m-(n_1-1))-1}{n_1-1}.
\end{equation}
By \eqref{eq18.-2}, we have
$|\pi_m(A_1)| \geq n-2\alpha n $
and also
$ m\leq (1+\beta)n;$
consequently,
\begin{equation}\label{eq18.4}
    2|\pi_m(A_1)|-m \geq (1-4\alpha-\beta)n \geq (1-4\alpha-\beta)n_1.
\end{equation}
From \eqref{eq18.2}, \eqref{eq18.3} and \eqref{eq18.4} we have
\begin{equation}\label{eq18.4.5}
    \frac{(1-4\alpha-\beta)n_1(m+1-n_1)-1}{n_1-1}\leq \varepsilon n_1,
\end{equation}
and so our bounds on the parameters imply
\begin{eqnarray*}
        m+1&\leq& n_1+ (1-4\alpha-\beta)^{-1}(\varepsilon (n_1-1)+\frac{1}{n_1})
        \leq n+(1-4\alpha-\beta)^{-1}(\varepsilon n+\frac{(1-\alpha)^{-1}}{n})\\
        &\leq& n+(1-4\alpha-\beta)^{-1}(\varepsilon n+\varepsilon^2(1-\alpha)^{-1}n)\\
        =(1+(1-4\alpha-\beta)^{-1}(\varepsilon+\varepsilon^2(1-\alpha)^{-1}))n.
\end{eqnarray*}
Finally, we conclude the arithmetic progression $Q'=\{0,d,\hdots, md\}$ has common difference $d$ and size at most $\lfloor (1+\nu)n \rfloor$ and contains $B$.\\

\textbf{The construction of $P'$.}
By translating $A_1$, we may assume that the partition
$A_1=A_l\sqcup A_c \sqcup A_r$
given by
$$A_l=(-\infty, -m) \cap A_1 \text{ , } A_c=[-m,0)\cap A_1 \text{ , } A_r=[0, \infty)\cap A_1.$$
satisfies
\begin{equation}\label{eq18.45}
    |A_r|=k \text{ , } |A_l|=k-\tau \text{ , } |A_c|=n_1-2k+\tau,
\end{equation}
where
$ \tau \in \{0,1\}.$
By \eqref{eq18.-2} and \eqref{eq18.1} we have $k \leq 2^{-10}n_1$, and
by \eqref{eq18.0} and \eqref{eq18.1} we have $2\varepsilon n_1 \geq 1$.
However, we may assume that
\begin{equation}\label{eq18.5}
    2^{-10}n_1\geq k \geq 2^{10}(\beta+\varepsilon) n_1 \geq 2^9,
\end{equation}
and so Theorem~\ref{technical_three_translates_Z} implies
$${\mathbb E}_{b_2\in B_1 \setminus \{m\}}\bigg|A_c +\{0,b_2,m\}\bigg| \ge 2(n_1-2k+\tau) + (n_1-2k+\tau) \frac{2k-\tau-1}{n_1-1}.$$
In particular, there exist $b_2 \in B_1\setminus\{m\}$ and there exists a partition
$ [-m,m)=Y\sqcup Y^c  $
such that
\[
        |Y|=2m-[2(n_1-2k+\tau)+(n_1-2k+\tau)\frac{2k-\tau-1}{n_1-1}] \geq 2k-\tau-1,
\]
\begin{eqnarray}\label{eq18.8}
|Y^c|&=& 2(n_1-2k+\tau)+(n_1-2k+\tau)\frac{2k-\tau-1}{n_1-1}\\
&\geq& 2n_1-1 -2k+\tau -\frac{(2k-\tau-1)^2}{n_1-1},
\end{eqnarray}
and
\[
        Y^c \subset A_c \cup (A_c+m) \cup (A_c+b_2).
\]

\vspace{5pt}
We now break the argument into two cases.

\vspace{5pt}
{\bf Case 1:}
\begin{equation}\label{eq18.9}
    \bigg| [-m-k,-m)\cap A_1\bigg|\geq k/2 \text{ and }  \bigg| [0,k)\cap A_1\bigg|\geq k/2.
\end{equation}
In this case, by \eqref{eq18.8} we get
$$\bigg|[-m,0)\cap Y  \bigg| \geq k-1 \text{ or } \bigg|[0,m)\cap Y  \bigg|\geq k-1.$$
Let
$Z=Y\cap [-m,0).$
We may assume that
$|Z|\geq k-1.$
By \eqref{eq18.5} we have
\begin{equation}\label{eq18.10}
    |Z|\geq k/2.
\end{equation}

\vspace{5pt}
{\bf Claim A.}{\em
$$ {\mathbb E}_{b_3 \in [0,4k)\cap B_1}\bigg|(A_1+b_3)\cap Z\bigg| \geq k/32.$$
In particular, there exists $b_3 \in B_1$ such that $$\bigg|(A_1+b_3)\cap Z\bigg| \geq k/32.$$
}
\begin{proof}
By \eqref{eq18.-2}, \eqref{eq18.25} and \eqref{eq18.5} we have
\begin{equation}\label{eq18.11}
        \bigg|[0,4k) \setminus B_1  \bigg| \leq \bigg|[0,m) \setminus B  \bigg| +|B_2|
        \leq \beta n +n_2
        \leq 4(\beta+\varepsilon) n_1
        \leq k/4.
\end{equation}
For $x\in [-m+3k,0)$, by \eqref{eq18.-2}, \eqref{eq18.25}, \eqref{eq18.45} and \eqref{eq18.5} we have
\begin{eqnarray}\label{eq18.13}
        \bigg|(x-4k,x]\setminus A_1 \bigg| &\leq& \bigg|(x-3k,x]\setminus A_1 \bigg|+k
        \leq \bigg|[-m,0)\setminus A_1 \bigg|+k \\
        &\leq& m-(n_1-2k+\tau)+k
        \leq \beta n +n_2 + 3k \\
        &\leq& 4(\beta+\varepsilon) n_1 + 3k
        \leq  7k/2.
\end{eqnarray}
For $x\in [-m,-m+3k)$, by \eqref{eq18.9} we get
\begin{equation}\label{eq18.12}
\bigg|(x-4k,x]\setminus A_1 \bigg| \leq 7k/2.
\end{equation}
For $z\in [-m,0)$, by \eqref{eq18.11}, \eqref{eq18.13}, \eqref{eq18.12} we get
\begin{equation}\label{eq18.14}
    {\mathbb P}_{b_3\in [0,4k)\cap B_1}\bigg(z \in A_1+b_3\bigg)\geq 1/16.
\end{equation}
Therefore, by \eqref{eq18.10} and \eqref{eq18.14} we conclude
$${\mathbb E}_{b_3 \in [0,4k)\cap B_1}\bigg|(A_1+b_3)\cap Z\bigg|= \sum_{z\in Z}{\mathbb P}_{b_3\in [0,4k)\cap B_1}\bigg(z \in A_1+b_3\bigg) \geq k/32. $$

\vspace{5pt}
This concludes the proof of Claim A.
\end{proof}

\vspace{5pt}
We now turn to the second case.

\vspace{5pt}
{\bf Case 2:}
\begin{equation*}
    \bigg| [-m-k,-m)\cap A_1\bigg|\leq k/2 \text{ or } \bigg| [0,k)\cap A_1\bigg|\leq k/2.
\end{equation*}
In this case we may assume
\begin{equation}\label{eq18.15}
    \bigg| [-m-k,-m)\cap A_1\bigg|\leq k/2 \text{ and } \bigg| [-\infty,-m-k)\cap A_1\bigg|\geq k/4.
\end{equation}
Put
$ Z=(-\infty,-m)\setminus A_1.$

\vspace{5pt}
{\bf Claim B.}{\em
$$ {\mathbb E}_{b_3 \in [0,k)\cap B_1}\bigg|(A_1+b_3)\cap Z\bigg| \geq k/32.$$
In particular, there exists $b_3 \in B_1$ such that $\bigg|(A_1+b_3)\cap Z\bigg| \geq k/32.$
}
\begin{proof}
Let $S$ be set of the greatest $k/4$ elements of $[-\infty,-m-k)\cap A_1$.\\

\vspace{5pt}
For $x\in S$, by \eqref{eq18.15} we have
\begin{equation}\label{eq18.16}
    \bigg|[x,x+k)\cap Z \bigg| \geq k/4.
\end{equation}
By \eqref{eq18.-2}, \eqref{eq18.25} and \eqref{eq18.5} we get
\begin{equation}\label{eq18.17}
        \bigg|[0,k) \setminus B_1  \bigg| \geq k-\bigg|[0,m) \setminus B  \bigg| -|B_2|
        \geq k-\beta n -n_2
        \geq k-4(\beta+\varepsilon) n_1
        \geq 7k/8.
\end{equation}
For $x\in S$, by \eqref{eq18.16} and \eqref{eq18.17} we have ${\mathbb P}_{b_3\in [0,k)\cap B_1}\bigg(x+b_3\in Z\bigg)\geq 1/8$,
so
$${\mathbb E}_{b_3 \in [0,k)\cap B_1}\bigg|(A_1+b_3)\cap Z\bigg|= \sum_{x\in S}{\mathbb P}_{b_3\in [0,k)\cap B_1}\bigg(x+b_3 \in Z\bigg) \geq k/32. $$

\vspace{5pt}
This ends the proof of Claim B.
\end{proof}

\vspace{5pt}
We now return to the proof of Theorem~\ref{stab_technical}.
By construction, we have
$$ A_1+\{0,b_2,b_3,m\} \supset$$
$$A_l \sqcup (A_r+m) \sqcup \bigg( [A_c \cup (A_c+b_2) \cup (A_c+m)]\cap Y^c \bigg) \sqcup \bigg((A_1+b_3)\cap Z \bigg).$$
So by Claim A and Claim B, together with  \eqref{eq18.27}, \eqref{eq18.45}, \eqref{eq18.5} and \eqref{eq18.8},
we obtain
\begin{eqnarray*}
        2n_1-1+8\varepsilon n_1&\geq& k+(k-\tau) + (2n_1-1 -2k+\tau -\frac{(2k-\tau-1)^2}{n_1-1}) + \frac{k}{32}\\
        &\geq& 2n_1-1+\frac{k}{32}-\frac{(2k-\tau-1)^2}{n_1-1} \geq  2n_1-1+\frac{k}{32}-\frac{4k^2}{n_1}\\
        &\geq& 2n_1-1+\frac{k}{64}.
\end{eqnarray*}
Therefore we have
$ k \leq 2^9 \varepsilon n_1.$

Finally, we conclude the arithmetic progression $P'=\{0,d,\hdots, md\}$ with common difference $d$ and size at most $\lfloor(1+\beta)n\rfloor\leq \lfloor (1+\mu)n \rfloor$ satisfies $|A\setminus P'| \leq n_2+2k \leq 2^{11}\varepsilon n \leq \mu n.$

\vspace{5pt}
This finishes the proof of Theorem~\ref{stab_technical}.
\end{proof}

\vspace{5pt}
We now prove Theorem~\ref{weak_linear_stability_four_translates}.
\begin{proof} [Proof of Theorem \ref{weak_linear_stability_four_translates}]
  Fix $\delta<2^{-10}$ and choose $\varepsilon = 2^{-40} \delta^8$. By Theorem \ref{equality_four_translates}, we may assume that $\varepsilon n \geq 1$. In particular, we have $\varepsilon^{1/2}n \geq 2^{20}$ and $\varepsilon^{1/8} \leq 2^{-10}$. We may also assume that $\min(B)=0$ and $\max(B)=m$.

\vspace{5pt}
The proof will proceed via the next three lemmas.

\begin{lemma}\label{1-16}
$$|\pi_m(A)|\geq (1-4\varepsilon^{1/4})n.$$
\end{lemma}

\begin{proof}
The proof of this lemma is based on two claims. Suppose for a contradiction that
\begin{equation*}
    |\widetilde{A}|\leq (1-4\varepsilon^{1/4})n.
\end{equation*}

\vspace{7pt}
{\bf Claim A.} {\em
There exists $b_1 \in B$ such that $$\bigg|A+\{0,b_1,m\}\bigg|\geq 2n-1-2^{-1}\varepsilon^{1/2} n$$ and $$\bigg|\pi_m[A+\{0,b_1,m\}]\bigg|\leq \big(1-\varepsilon^{1/2}\big) n.$$
}

\begin{proof}
By Theorem \ref{strengthening_of_three_translates} we have
$$ {\mathbb E}_{b\in B\setminus\{m\}}\bigg|A+\{0,b,m\}\bigg| \geq 2n-1.$$
By hypothesis we have
$$\max_{b\in B} \bigg|A+\{0,b,m\}\bigg|\leq 2n-1+\varepsilon n.$$
By Markov's inequality, it follows that
\begin{equation}\label{eq13.1}
{\mathbb P}_{b\in B\setminus \{m\}} \bigg( \bigg|A+\{0,b,m\}\bigg| \geq 2n-1-2^{-1}\varepsilon^{1/2} n \bigg)>1-4\varepsilon^{1/2}.
\end{equation}
Now form the partition
$A=A_1\sqcup A_2$
such that
\begin{equation}\label{eq13.15}
A_1^x=
\begin{cases}
A^x \text{ if } |A^x|\geq 2,\\
\emptyset \text{ if } |A^x|\leq 1.
\end{cases}
\text{ and } A_2^x=
\begin{cases}
A^x \text{ if } |A^x|\leq 1,\\
\emptyset \text{ if } |A^x|\geq 2.
\end{cases}
\end{equation}
From    $|\widetilde{A}|\leq (1-4\varepsilon^{1/4})n$
we get  $|A_1| \geq 4\varepsilon^{1/4}n$.
Lemma~\ref{lem8.2} gives
\begin{eqnarray*}
        {\mathbb E}_{b\in B\setminus\{m\}}\bigg|(A_1+b)\setminus \pi^{-1}(\widetilde{A})\bigg|
        &\geq& |A_1|\max\bigg(0,\frac{|\widetilde{B}| - |\widetilde{A}|}{|\widetilde{B}|}\bigg)\\
         &\geq& 4\varepsilon^{1/4}n \frac{n-1-(1-4\varepsilon^{1/4})n}{n-1}
         \geq 8 \varepsilon^{1/2}n.
\end{eqnarray*}
Trivially, we have
$$ \max_{b\in B\setminus \{m\}} \bigg|(A_1+b)\setminus \pi^{-1}(\widetilde{A})\bigg| \leq n.$$
By Markov's inequality, we deduce
\begin{equation}\label{eq13.2}
 {\mathbb P}_{b\in B \setminus \{m\}} \bigg( \bigg|(A_1+b)\setminus \pi^{-1}(\widetilde{A}) \bigg| \geq 4\varepsilon^{1/2} n  \bigg) > 4\varepsilon^{1/2} .
\end{equation}
By \eqref{eq13.1} and \eqref{eq13.2}, it follows that there exists $b_1 \in B \setminus \{m\}$ such that
$$ \bigg|A+\{0,b_1,m\}\bigg| \geq 2n-1-2^{-1}\varepsilon^{1/2}n $$
and
\begin{equation}\label{eq13.3}
    \bigg|(A_1+b_1)\setminus \pi^{-1}(\widetilde{A})  \bigg| \geq 4\varepsilon^{1/2} n.
\end{equation}
By \eqref{new-sum-eq}, together with \eqref{eq13.0}, \eqref{eq13.15} and \eqref{eq13.3}, it follows that
\begin{eqnarray*}
        2n-1+\varepsilon n \geq \bigg|A+\{0,b_1,m\}\bigg|
        &\geq& \bigg|A\cup (A+m) \bigg| +\bigg|[A+b_1] \setminus [A\cup (A+m)]\bigg|\\
        &\geq& |A|+|\widetilde{A}|+\bigg|(A+b_1)\setminus \pi_m^{-1}(\widetilde{A})\bigg|\\
        &\geq& |A|+|\widetilde{A}|+\bigg|(\widetilde{A}+\widetilde{b}_1)\setminus \widetilde{A}\bigg|+\frac{1}{2}\bigg|(A_1+b_1)\setminus \pi_m^{-1}(\widetilde{A})\bigg|\\
        &\geq& (1+2\varepsilon^{1/2})n+\bigg|\widetilde{A}\cup (\widetilde{A}+\widetilde{b}_1)\bigg|.
\end{eqnarray*}
Hence
$$\bigg|\pi_m[A+\{0,b_1,m\}]\bigg|= \bigg|\widetilde{A}\cup (\widetilde{A}+\widetilde{b}_1)\bigg| \leq \big(1-\varepsilon^{1/2}\big) n.$$

\vspace{5pt}
This establishes Claim A.
\end{proof}

\vspace{7pt}
{\bf Claim B.} {\em
$${\mathbb E}_{b_2\in B\setminus \{m\}} \bigg|A+\{0,b_1, b_2,m\}\bigg|>2n-1+\varepsilon n.$$
}

\begin{proof}
From the first part of Claim A we get
$$\bigg|A+\{0,b_1,m\}\bigg|\geq 2n-1-2^{-1}\varepsilon^{1/2} n$$
From the second part of Claim A and Lemma~\ref{lem8.2}  we get
\begin{eqnarray*}
        {\mathbb E}_{b_2\in B\setminus \{m\}} \bigg|[A+b_2]\setminus [A+\{0,b_1,m\}]\bigg| &\geq& {\mathbb E}_{b_2\in B\setminus \{m\}}\bigg|[A+b_2]\setminus \pi_m^{-1}[\widetilde{A}\cup (\widetilde{A}+\widetilde{b}_1)]\bigg|\\
        &\geq& |A|\max\bigg(0, \frac{|\widetilde{B}|- |\widetilde{A}\cup (\widetilde{A}+\widetilde{b}_1)|}{|\widetilde{B}|} \bigg)\\
        &\geq&  n \max \bigg(0, \frac{n-1-(1-\varepsilon^{1/2})n}{n-1} \bigg)
        \geq \frac{3}{4}\varepsilon^{1/2}n.
\end{eqnarray*}
Combining the inequalities above, we obtain
$${\mathbb E}_{b_2\in B\setminus \{m\}} \bigg|A+\{0,b_1, b_2,m\}\bigg|\geq 2n-1+\frac{1}{4}\varepsilon^{1/2}n>2n-1+\varepsilon n.$$

\vspace{5pt}
This proves Claim B.
\end{proof}
This completes the proof of Lemma~\ref{1-16}, as we see that we have our desired contradiction.
\end{proof}

\vspace{5pt}
We now move on to our second lemma.

\begin{lemma}\label{lem11.2}
  There exists an arithmetic progression $Q$ of size $\lfloor (1+32\varepsilon^{1/4})n\rfloor$ such that
  $B \subset Q$.
\end{lemma}

\begin{proof}
The proof of this lemma is based on the following two claims.

\vspace{7pt}
{\bf Claim A.} {\em
For every $\widetilde{c} \in \widetilde{B}+ \widetilde{B}$ there exists a set $\widetilde{A}' \subset \widetilde{A}$ such that $$ \widetilde{A}'+\widetilde{c}\subset \widetilde{A} \text{ and } |\widetilde{A}'|\geq (1-10\varepsilon^{1/4})n.$$
}
\begin{proof}
Let $b_1, b_2 \in B$
such that
$\widetilde{c}=\widetilde{b}_1+ \widetilde{b}_2$,
and fix
$b\in \{b_1,b_2\}$.
By \eqref{eq13.0}, \eqref{new-sum-eq} and Lemma~\ref{1-16} we have
\begin{eqnarray*}
        2n-1+\varepsilon n &\geq& \bigg|A+\{0,b,m\}  \bigg|
        \geq \bigg|A\cup (A+m)\bigg|+ \bigg|[A+b]\setminus [A\cup (A+m)] \bigg|\\
        &\geq& |A|+|\widetilde{A}| + \bigg| (\widetilde{A}+\widetilde{b})\setminus \widetilde{A} \bigg|
        \geq (2-4\varepsilon^{1/4})n +\bigg| (\widetilde{A}+\widetilde{b})\setminus \widetilde{A} \bigg|.
\end{eqnarray*}
Therefore
$ \bigg| (\widetilde{A}+\widetilde{b})\setminus \widetilde{A} \bigg| \leq 5\varepsilon^{1/4}n,$
and in particular, we have
$ \bigg| (\widetilde{A}+\widetilde{c})\setminus \widetilde{A} \bigg| \leq 10\varepsilon^{1/4}n.$
We conclude that the set $\widetilde{A}'=\{\widetilde{a}\in \widetilde{A} \text{ : } \widetilde{a}+\widetilde{c}\in \widetilde{A} \}$
has size $|\widetilde{A}'|\geq (1-10\varepsilon^{1/4})n$, so Claim A is proved.
\end{proof}

\vspace{7pt}
{\bf Claim B.} {\em
$$|\widetilde{B}+\widetilde{B}|\leq (1+32\varepsilon^{1/4})|\widetilde{B}|.$$
}

\begin{proof}
Recall that $|\widetilde{A}| \leq n \text{ and } |\widetilde{B}|=n-1.$
Consider the set
\[
E(\widetilde{A})=\{(a_1,a_2,a_3,a_4)\in \widetilde{A}^4 \text{ : } a_1-a_2=a_3-a_4\}.
\]
Trivially,
$ |E(\widetilde{A})| \leq n^3.$
On the other hand, by the previous claim we have
$|E(\widetilde{A})|\geq (1-10\varepsilon^{1/4})^2n^2|\widetilde{B}+\widetilde{B}|.$
This completes the proof of Claim B.
\end{proof}

\vspace{5pt}
Returning to the proof of Lemma~\ref{lem11.2}, by Claim B we have $|\widetilde{B}+\widetilde{B}|\leq (1+32\varepsilon^{1/4})|\widetilde{B}|.$
Kneser's inequality \cite{Knes} now implies that $\widetilde{B}+\widetilde{B}$ is a union of cosets of a subgroup $\widetilde{H}$ of $\mathbb{Z}_m$  with $|\widetilde{H}|\geq (1-32\varepsilon^{1/4})|\widetilde{B}|.$
Because $0 \in \widetilde{B}$, it follows that
$\widetilde{B}+\widetilde{B}=\widetilde{H}$ where $|\widetilde{H}| \leq (1+32\varepsilon^{1/4})|\widetilde{B}|$ \ and \ $|\widetilde{B} \setminus \widetilde{H}|=0.$
If we let $Q$ be the unique arithmetic progression in $\mathbb{Z}$ satisfying
$ \{0,m\} \subset Q \subset [0,m]$ and  $\pi_m(Q)=\widetilde{H}$
then we can conclude that
$ |Q|=1+|\widetilde{H}| \leq (1+32\varepsilon^{1/4})n  \text{ and } B \subset Q.$
This finishes the proof of Lemma~\ref{lem11.2}.
\end{proof}

\vspace{5pt}
Our final lemma is as follows.

\begin{lemma}\label{weak_A}
There exists a translate $P$ of $Q$ such that $|A\setminus P|\leq 16\varepsilon^{1/8}n.$
\end{lemma}
\begin{proof}
Let $X$ be the set of the first $\varepsilon^{1/8}n$ elements of $A$, and assume that $$0=\min(B)=\min(Q) \text{ and } m=\max(B)=\max(Q).$$
Suppose for a contradiction that for every $x\in X$ we have $|(x+Q) \cap A|\leq (1-16\varepsilon^{1/8})n,$
so, in particular,
$|(x+B) \cap A|\leq (1-16\varepsilon^{1/8})n.$
On the one hand, from this assumption we get
\begin{eqnarray*}
        &&{\mathbb E}_{b \in B} \bigg( [X+b] \setminus [A\cup (A+m)]\bigg) = \sum_{x\in X} {\mathbb P}_{b \in B} \bigg( x+b \not\in A\cup (A+m) \bigg)\\
        &=& \sum_{x\in X} {\mathbb P}_{b \in B} \bigg( x+b \not\in A\cup (X+m) \bigg)
        = \sum_{x\in X}  \frac{\bigg|[x+B]\setminus [A\cup(X+m)]\bigg|}{|B|} \\
        &\geq& \sum_{x\in X} \max\bigg(0, \frac{|B|-|(x+B)\cap A|-|X|}{|B|} \bigg)\\
        &\geq& \varepsilon^{1/8}n\max\bigg( 0, \frac{n-(1-16\varepsilon^{1/8})n-\varepsilon^{1/8}n}{n}\bigg)
        = 15\varepsilon^{1/4}n.
\end{eqnarray*}
On the other hand, from \eqref{eq13.0}, \eqref{new-sum-eq} and Lemma~\ref{1-16} we have
\begin{eqnarray*}
        2n-1+\varepsilon n \geq \max_{b\in B} \bigg|A+\{0,b,m\} \bigg|
        &\geq& \bigg|A\cup (A+m)  \bigg|+ \max_{b\in B} \bigg| [A+b] \setminus [A\cup (A+m)]  \bigg|\\
        &\geq& |A|+|\widetilde{A}| + \max_{b\in B}\bigg| [A+b] \setminus [A\cup (A+m)]  \bigg|\\
        &\geq& (2-4\varepsilon^{1/4})n +\max_{b\in B} \bigg| [A+b] \setminus [A\cup (A+m)]  \bigg|.
\end{eqnarray*}
This is a contradiction, as desired, completing our proof.
\end{proof}

\vspace{5pt}
With this, we have  proved Theorem~\ref{weak_linear_stability_four_translates}, since it is implied by
Lemma~\ref{lem11.2} and Lemma~\ref{weak_A}.
\end{proof}

\vspace{5pt}
To end this section, we prove Theorem~\ref{linear_stability_four_translates}.

\begin{proof}[Proof of Theorem \ref{linear_stability_four_translates}]
The strategy is to first apply Theorem~\ref{weak_linear_stability_four_translates}, and then apply Theorem~\ref{stab_technical} twice. Consider the function $\delta(\varepsilon)$ given by Theorem~\ref{weak_linear_stability_four_translates}, and the functions $\nu(\alpha, \beta, \varepsilon), \mu(\alpha, \beta, \varepsilon)$ given by Theorem~\ref{stab_technical}. It is easy to check that
$\delta \rightarrow 0 \text{ as } \varepsilon \rightarrow 0,$
and
$  \mu, \nu\rightarrow 0 \text{ as } \alpha, \beta, \varepsilon \rightarrow 0.$
It is also a simple computational task to find $c>0$ such that
$ \mu_1=\mu(\delta (\varepsilon), \delta(\varepsilon), \varepsilon)$,  $\nu_1=\nu(\delta (\varepsilon), \delta(\varepsilon), \varepsilon)$,
$ \mu_2=\mu(\mu_1,\nu_1,\varepsilon)$ and  $\nu_2=\nu(\mu_1,\nu_1,\varepsilon)$
satisfy
$$ \mu_2\leq c\varepsilon \text{ and } \nu_2 \leq \varepsilon+c\varepsilon^2 \text{ as } \varepsilon \rightarrow 0.$$
The result follows from these estimates.
\end{proof}

\vspace{5pt}
To end the paper, we give one of the many possible conjectures.

\vspace{5pt}
Our inverse theorem about $\mathbb{Z}$,
Theorem \ref{linear_stability_four_translates},
involves taking four elements from $B$ instead
of three elements. We believe that a similar result should hold with three elements. We state this as a
conjecture in the weaker form corresponding to
Theorem \ref{weak_linear_stability_four_translates} instead of
Theorem \ref{linear_stability_four_translates},
although we do believe that the stronger form should be true as well.

\vspace{5pt}
\begin{conjecture}\label{conj-three}
For every $\delta >0$ there exists $\varepsilon > 0$ such that
the following is true. Suppose that $A$ and $B$ are finite subsets of
${\mathbb Z}$ of equal size $n$ such that for any
elements $b_1,b_2,b_3 \in B$
we have $$|A+\{b_1,b_2,b_3\}| \leq (2+ \varepsilon) n - 1.$$ Then there exist arithmetic progressions $P,Q$ in
$\mathbb{Z}$ with the same common difference such that 
\[
B\subset Q \text{ and } |A \Delta P|,
|B\Delta Q| \leq \delta n.
\]
\end{conjecture}

We remark that if this conjecture is proved then, by virtue of
Theorem~\ref{strengthening_of_three_translates}, it would actually yield a strengthening of
itself in which $|B \Delta Q| \leq (1+o(1))\varepsilon n$.
However, as we remarked earlier, it is important to note, for Conjecture 3,
that the selection of three elements {\it cannot} be
done as in Theorem
\ref{intro_three_translates_Z}, by taking the maximum and minimum elements of $B$ plus one more
element.


\begin{thebibliography}{99}


\bibitem{BLT} 
Bollob\'as, B., I. Leader and M. Tiba, Large sumsets from small sumsets, arXiv:2204.07559, 52~pp.

\bibitem{BGT-doubling}
Breuillard, E., B.~Green and T.~Tao,
Small doubling in groups, in {\em Erd\H{o}s Centennial},  Bolyai Soc. Math. Stud., {\bf 25}, J\'anos Bolyai Math. Soc., Budapest, 2013, pp.~129--151.

\bibitem{Cauchy}
Cauchy,~A., Recherches sur les nombres, {\em J. \'Ecole Polytechnique} {\bf 9} (1813) 99--116. Also de Cauchy,~A.L., Recherches sur les nombres, {\OE}uvres (2nd series, Paris, 1882), vol. {\bf 1}, pp.~39--63.

\bibitem{Dav1}
Davenport,~H., On the addition of residue classes, {\em J. London Math. Soc.} {\bf 10} (1935)
30--32.

\bibitem{Dav2}
Davenport, H., A historical note, {\em J. London Math. Soc.} {\bf 22} (1947) 100--101.

\bibitem{Frei-59}
Fre\v{\i}man, G.A., The addition of finite sets I,  {\em Izv. Vys\v{s}. U\v{c}ebn. Zaved. Matematika} {\bf 6} (1959)  202--213. (in Russian)


\bibitem{Frei-62}
Fre\v{\i}man, G.A.,
Inverse problems of additive number theory VI, On the addition of finite sets III,  {\em Izv. Vys\v{s}. U\v{c}ebn. Zaved. Matematika} {\bf 28} (1962) 151--157. (in Russian)


\bibitem{grynkiewicz-2}
Grynkiewicz,~D. {\em Structural additive theory}. \textbf{30}. Springer Science and Business Media, 2013, xii+426~pp.


\bibitem{Knes}
Kneser, M., Summenmengen in lokalkompakten abelschen Gruppen,  {\em Math. Z.} {\bf 66} (1956) 88--110.


\bibitem{LevSme}
Lev,~V.F., and P.Y.~Smeliansky,
On addition of two distinct sets of integers,
{\em Acta Arith.} {\bf 70} (1995) 85--91.


\bibitem{Nathbook}
Nathanson,~M.B.,
{\em Additive Number Theory.
  Inverse Problems and the Geometry of Sumsets}, Graduate Texts in Mathematics {\bf 165}, Springer-Verlag, New York,
1996, xiv+293~pp.


\bibitem{Stan}
Stanchescu, Y., On addition of two distinct sets of integers,
{\em Acta Arith.} {\bf 75} (1996) 191--194.


\bibitem{taobook}
Tao,~T.\, and V.H.~Vu,
{\em Additive Combinatorics},
Cambridge Studies in Advanced Mathematics {\bf 105},
Cambridge University Press, Cambridge, 2006, xviii+512~pp.


\end{thebibliography}
\end{document}